\documentclass[11pt]{article}
\usepackage[margin=1in]{geometry}
\usepackage{amsmath,amssymb,amsthm,mathtools}
\usepackage{microtype}
\usepackage{booktabs}
\usepackage{tikz}
\usepackage[round,authoryear]{natbib}
\usepackage[hidelinks]{hyperref}
\usepackage[nameinlink,capitalise,noabbrev]{cleveref}

\newtheorem{theorem}{Theorem}[section]
\newtheorem{lemma}[theorem]{Lemma}
\newtheorem{corollary}[theorem]{Corollary}

\newcommand{\Deck}{\mathcal D}

\newcommand{\Even}{\operatorname{Even}}

\title{Nonisomorphic Graphs Can Share an Arbitrarily Large Fraction of Their Vertex-Deleted Cards}
\author{Sergey Ivanov}
\date{}

\begin{document}
\maketitle

\begin{abstract}
For a graph $G$, its vertex deck is the multiset of graphs obtained by deleting one vertex.
Bowler, Brown, and Fenner (BBF) proposed $2\lfloor(n-1)/3\rfloor$ as the maximum possible overlap
between the decks of two nonisomorphic $n$-vertex graphs, for all sufficiently large $n$.
We first give an explicit pair of connected nonisomorphic graphs on $78$ vertices with at least
$51$ common cards, exceeding BBF's predicted value of $50$. We then construct, for every even
$r\ge4$, families at arbitrarily large orders whose overlap fraction is asymptotically at least
$1-1/r$. Consequently, for every $\alpha<1$, infinitely many pairs have more than $\alpha n$ common
cards, so the attainable fraction is arbitrarily close to the full deck. For representative instances,
the predicted overlaps were also checked by complete deck generation and isomorphism testing with
Brendan McKay's nauty tools.\footnote{Code and verification artifacts:
\url{https://github.com/nd7141/reconstruction-conjecture-counterexample}.}
\end{abstract}

\section{Introduction}\label{sec:intro}
For a finite simple graph $G$ and $v\in V(G)$, the graph $G-v$ is a \emph{card} of $G$, and
\[
  \Deck(G)=\{G-v:v\in V(G)\}
\]
is its vertex deck, with multiplicity. For graphs $G,H$ of the same order, let
\[
 b(G,H)=|\Deck(G)\cap\Deck(H)|,
 \qquad
 B(n)=\max_{\substack{|V(G)|=|V(H)|=n\\G\not\cong H}} b(G,H),
\]
where the intersection also respects multiplicity. Thus $b(G,H)$ is the maximum number of deletion
occurrences that can be paired to give isomorphic cards.

\citet{BBF10} constructed nonisomorphic $n$-vertex graphs with
\[
 Q(n)=2\left\lfloor\frac{n-1}{3}\right\rfloor
\]
common cards and proposed $Q(n)$ as the eventual unrestricted maximum. This assertion is called the
\emph{Strong Reconstruction Conjecture} in the reconstruction-number literature
\citep{AFLM10,BF20slides}. In this paper we show that the conjectured maximum is exceeded by an infinite
family and that common-card overlap can in fact approach the full deck. Specifically, for every
$\alpha<1$ we construct connected nonisomorphic pairs, of arbitrarily large order $n$ and with the same
degree sequence, for which $b(G,H)>\alpha n$.

The ordinary Reconstruction Conjecture is the separate assertion that the entire deck determines every
graph of order at least three; in the notation above it states $B(n)\le n-1$. Our results concern large
proper intersections of two decks and do not resolve that conjecture.

\section{The smallest counterexample: 78 vertices}\label{sec:78}
The smallest counterexample currently known to us has $78$ vertices; here ``counterexample'' means a
pair exceeding the BBF value $Q(78)$, not a proof that no smaller pair exists. The construction is best
understood in three steps.
\begin{enumerate}
\item Four ports are expanded into independent classes $P_1,\ldots,P_4$. Vertices in the same class
have identical neighbors outside the class; such a class is called a \emph{false-twin class}.
\item A fixed $12$-vertex selector allows the four classes to be rearranged by every even permutation,
but by no odd permutation.
\item The class sizes in $G$ and $H$ differ by one odd swap. The intact graphs are therefore
nonisomorphic. After a suitable vertex deletion, two class sizes become equal, making a second swap
invisible. The two swaps together are even, and the resulting cards are isomorphic.
\end{enumerate}

We now build the selector. Write $\Even(r)$ for the set of even permutations of $[r]$; this is the
group usually called the alternating group. Let $X=P_4$ on the labeled set $[4]$, and let
\[
 \mathcal F=\{\sigma X:\sigma\in\Even(4)\}.
\]
Reversing $P_4$ exchanges both end pairs and is therefore even. Since there are $12$ even
permutations of four labels and two symmetries of $P_4$, the set $\mathcal F$ contains six labeled
paths.

The selector has four ports $p_1,\dots,p_4$, one pair vertex $q_{ij}$ for each of the six unordered
pairs, and one selector vertex $z_F$ for each $F\in\mathcal F$. Join $p_i$ to each $q_{jk}$ with
$i\in\{j,k\}$. The six vertices $z_F$ form a clique. A selector vertex $z_F$ is joined to exactly the
three pair vertices $q_{ij}$ for which $ij$ is an edge of the path $F$. Numbering the six paths as in
\cref{fig:gadget}, their pair-vertex neighborhoods are
\begin{align*}
1:&\ \{q_{12},q_{13},q_{24}\}, &
2:&\ \{q_{12},q_{14},q_{34}\}, &
3:&\ \{q_{12},q_{23},q_{34}\},\\
4:&\ \{q_{13},q_{14},q_{23}\}, &
5:&\ \{q_{13},q_{24},q_{34}\}, &
6:&\ \{q_{14},q_{23},q_{24}\}.
\end{align*}
Thus the green vertices in the figure form $K_6$, but they are not all adjacent to every pair vertex:
each green vertex has three orange neighbors, and each orange vertex has three green neighbors.

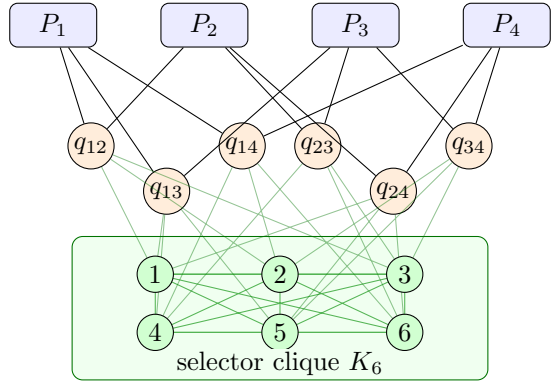
\begin{figure}[t]
\centering
\begin{tikzpicture}[
  font=\small,
  port/.style={draw,rounded corners=3pt,minimum width=1.15cm,minimum height=.58cm,fill=blue!8},
  pair/.style={circle,draw,minimum size=.61cm,inner sep=0pt,fill=orange!14},
  selector/.style={circle,draw,minimum size=.48cm,inner sep=0pt,fill=green!18}
]
\node[port] (p1) at (-3,3.05) {$P_1$};
\node[port] (p2) at (-1,3.05) {$P_2$};
\node[port] (p3) at (1,3.05) {$P_3$};
\node[port] (p4) at (3,3.05) {$P_4$};
\node[pair] (q12) at (-2.5,1.45) {$q_{12}$};
\node[pair] (q13) at (-1.5,.85) {$q_{13}$};
\node[pair] (q14) at (-.5,1.45) {$q_{14}$};
\node[pair] (q23) at (.5,1.45) {$q_{23}$};
\node[pair] (q24) at (1.5,.85) {$q_{24}$};
\node[pair] (q34) at (2.5,1.45) {$q_{34}$};
\draw (p1)--(q12) (p1)--(q13) (p1)--(q14);
\draw (p2)--(q12) (p2)--(q23) (p2)--(q24);
\draw (p3)--(q13) (p3)--(q23) (p3)--(q34);
\draw (p4)--(q14) (p4)--(q24) (p4)--(q34);
\filldraw[fill=green!6,draw=green!45!black,rounded corners=4pt]
  (-2.75,-1.65) rectangle (2.75,.25);
\coordinate (s1) at (-1.65,-.25); \coordinate (s2) at (0,-.25); \coordinate (s3) at (1.65,-.25);
\coordinate (s4) at (-1.65,-1.02); \coordinate (s5) at (0,-1.02); \coordinate (s6) at (1.65,-1.02);
\draw[green!45!black,opacity=.42]
 (s1)--(q12) (s1)--(q13) (s1)--(q24)
 (s2)--(q12) (s2)--(q14) (s2)--(q34)
 (s3)--(q12) (s3)--(q23) (s3)--(q34)
 (s4)--(q13) (s4)--(q14) (s4)--(q23)
 (s5)--(q13) (s5)--(q24) (s5)--(q34)
 (s6)--(q14) (s6)--(q23) (s6)--(q24);
\draw[green!55!black,opacity=.65]
 (s1)--(s2) (s1)--(s3) (s1)--(s4) (s1)--(s5) (s1)--(s6)
 (s2)--(s3) (s2)--(s4) (s2)--(s5) (s2)--(s6)
 (s3)--(s4) (s3)--(s5) (s3)--(s6)
 (s4)--(s5) (s4)--(s6) (s5)--(s6);
\node[selector] at (s1) {1}; \node[selector] at (s2) {2}; \node[selector] at (s3) {3};
\node[selector] at (s4) {4}; \node[selector] at (s5) {5}; \node[selector] at (s6) {6};
\node[fill=green!6,inner sep=1pt] at (0,-1.43) {selector clique $K_6$};
\end{tikzpicture}
\caption{Schematic structure of the $78$-vertex counterexample. Each blue node $P_i$ represents a
whole independent false-twin class, rather than a single vertex. Black edges encode which two classes
belong to each pair vertex. The six green vertices form a clique, and the green--orange edges encode
the six labeled paths listed above. The class sizes of $P_i$ are $(15,16,17,18)$ in $G$ and $(16,15,17,18)$ in
$H$.}
\label{fig:gadget}
\end{figure}

The purpose of the selector is visible: the orange layer records every pair of positions, while the
green layer records one orbit of paths obtained using only even relabelings. Any symmetry of the whole
gadget must preserve that orbit, so it cannot induce an odd relabeling of the ports.

\begin{lemma}[Permutations allowed by the selector]\label{lem:A4}
The port permutations induced by automorphisms of the selector are precisely the even permutations.
\end{lemma}
\begin{proof}
Ports, pair vertices, and selector vertices have degrees $3,5,$ and $8$, respectively, so an
automorphism preserves the three layers. Every even permutation preserves $\mathcal F$ and extends to
a symmetry of the entire selector. Conversely, let $\tau$ be the permutation induced on the four
ports. Pair incidence forces $q_{ij}$ to map to $q_{\tau(i)\tau(j)}$, so $\tau$ preserves $\mathcal F$.
Because $X\in\mathcal F$, there is an even permutation $\alpha$ such that $\tau X=\alpha X$. Hence
$\alpha^{-1}\tau$ is a symmetry of $P_4$, and both symmetries of $P_4$ are even. Therefore $\tau$ is
even.
\end{proof}

Replace each port $p_i$ by an independent false-twin class $P_i$, every vertex of which has the former
neighborhood of $p_i$. Let $G$ and $H$ have class-size vectors
\[
 (15,16,17,18)\qquad\text{and}\qquad(16,15,17,18),
\]
respectively, while sharing the same $12$ nonport vertices. Before the formal proof, observe why the
numbers are consecutive: deleting one vertex from a class can make it equal in size to the preceding
class. That equality supplies the harmless second swap that turns the forbidden odd swap into an
allowed even permutation.

\begin{theorem}\label{thm:78}
There are connected nonisomorphic graphs $G,H$ on $78$ vertices, with the same degree sequence, such that
\[
 b(G,H)\ge51>50=Q(78).
\]
\end{theorem}
\begin{proof}
The classes $P_i$ are exactly the nontrivial false-twin classes, so every isomorphism permutes them and,
by \cref{lem:A4}, induces an even permutation. Since all four sizes are distinct, an isomorphism from
$G$ to $H$ would have to induce the odd transposition $(1\;2)$. Thus $G\not\cong H$.

Deleting from the size-$16$ class in each graph gives the same vector $(15,15,17,18)$ and contributes
$16$ common cards. Deleting from $P_3$ in both graphs produces vectors $(15,16,16,18)$ and
$(16,15,16,18)$. The required transposition $(1\;2)$ can now be composed with the harmless transposition
$(2\;3)$ of equal classes, yielding an allowed even permutation; this contributes $17$ cards. Likewise,
a deletion from $P_4$ makes the last two class sizes equal, so $(1\;2)(3\;4)$ gives $18$ further cards.
Hence $b(G,H)\ge16+17+18=51$.

Both graphs are connected. Every port-class vertex has degree $3$, the selector degrees agree, and the
degree of $q_{ij}$ depends only on $|P_i|+|P_j|$. The two size vectors differ by a permutation, so their
multisets of pairwise sums, and therefore their degree sequences, are equal.
\end{proof}

\section{General construction and proof}\label{sec:general}
The four-port example already contains the complete mechanism. To make the common-card fraction closer
to $1$, we use more port classes while retaining a selector that distinguishes even from odd
permutations. Fix an even integer $r=2k\ge4$.

We first need a seed whose own symmetries are all even. On labeled vertices
$A=\{a_1,\ldots,a_k\}$ and $D=\{d_1,\ldots,d_k\}$, let $X_r$ consist of a clique on $A$, an independent
set $D$ (no two vertices of $D$ are adjacent), and the private edges $a_i d_i$. A symmetry can
arbitrarily permute the $k$ matched pairs
$(a_i,d_i)$, but it must apply the same permutation once on $A$ and once on $D$. Its sign is therefore
squared, so it is even on all $r$ labels. In particular, $X_r$ has $k!$ symmetries and all lie in
$\Even(r)$.

Let
\[
 \mathcal F_r=\{\sigma X_r:\sigma\in\Even(r)\}.
\]
Construct a selector $S_r$ with ports $p_i$, pair vertices $q_{ij}$, and a clique of vertices $z_F$
indexed by $F\in\mathcal F_r$. As before, port $p_i$ is adjacent to the pair vertices containing $i$,
and $z_F$ is adjacent to $q_{ij}$ exactly when $ij\in E(F)$. Since $X_r$ has $k!=(r/2)!$
symmetries,
\[
 |\mathcal F_r|=\frac{r!}{2(r/2)!},\qquad
 C_r=\binom r2+\frac{r!}{2(r/2)!}
\]
is the number of nonport selector vertices.

The next lemma is the general version of the four-port parity check. Its proof has two parts: degrees
identify the three layers, and incidence then forces every port permutation to preserve the chosen
family $\mathcal F_r$.

\begin{lemma}[Permutations allowed by the general selector]\label{lem:selector}
The port permutations induced by automorphisms of $S_r$ are precisely the even permutations of $[r]$.
\end{lemma}
\begin{proof}
First, the permutations preserving $\mathcal F_r$ are exactly the even ones. Indeed, every even
permutation preserves the family by definition. Conversely, if a permutation $\tau$ preserves
$\mathcal F_r$, then $\tau X_r=\alpha X_r$ for some even $\alpha$. Thus
$\alpha^{-1}\tau$ is a symmetry of $X_r$ and is even by the preceding construction, so $\tau$ is even.

It remains to recognize the layers inside $S_r$. Ports have degree $r-1$. Every pair vertex occurs in
the same number $\lambda$ of graphs in $\mathcal F_r$, so its degree is $2+\lambda$, while a selector
vertex has degree $|\mathcal F_r|-1+|E(X_r)|$. For $r=4$ these degrees are $3,5,8$. For $r=2k\ge6$,
\[
 \lambda=|\mathcal F_r|\frac{k+1}{2(2k-1)},
\]
which gives $2+\lambda>r-1$; also $\lambda<|\mathcal F_r|$ and $|E(X_r)|\ge3$, so selector vertices
have still larger degree. An automorphism therefore preserves all three layers. Port incidence forces
$q_{ij}\mapsto q_{\sigma(i)\sigma(j)}$, and selector incidence forces $F\mapsto\sigma F$. Hence
$\sigma$ preserves $\mathcal F_r$ and is even. Conversely, every even permutation extends to all three
layers.
\end{proof}

We now store numerical information in the ports. For a vector
$\mathbf a=(a_1,\ldots,a_r)$ with every $a_i\ge2$, let $S_r(\mathbf a)$ replace $p_i$ by an independent
false-twin class $P_i$ of size $a_i$. Because these classes are structurally recognizable, an
isomorphism must match their sizes using a permutation allowed by the selector.

\begin{lemma}[Recognizing the blown-up ports]\label{lem:blowup}
$S_r(\mathbf a)\cong S_r(\mathbf b)$ if and only if an even permutation $\sigma$ satisfies
$a_i=b_{\sigma(i)}$ for every $i$.
\end{lemma}
\begin{proof}
The $P_i$ are the only nontrivial false-twin classes: distinct classes have distinct pair-vertex
neighborhoods, pair vertices are identified by their incident classes, and selector vertices have
distinct pair-vertex neighborhoods. Any isomorphism therefore permutes the $P_i$. Contracting them
recovers $S_r$, so \cref{lem:selector} makes the induced permutation even. Conversely, an allowed even
permutation extends by arbitrary bijections between corresponding classes.
\end{proof}

The final size pattern uses consecutive integers and swaps only the first two entries. This makes the
intact assignment odd, but a deletion from almost any later class creates a repeated size and hence a
second, harmless transposition.

For $t\ge1$, define
\[
 \mathbf a(t)=(t+1,t+2,\ldots,t+r),\qquad
 \mathbf b(t)=(t+2,t+1,t+3,\ldots,t+r),
\]
and set $G_{r,t}=S_r(\mathbf a(t))$ and $H_{r,t}=S_r(\mathbf b(t))$.

\begin{theorem}[Fixed-$r$ overlap bound]\label{thm:main}
For every even $r\ge4$ and $t\ge1$, the graphs $G_{r,t},H_{r,t}$ are connected, nonisomorphic, have the
same degree sequence, and have
\[
 n_{r,t}=rt+\frac{r(r+1)}2+C_r
\]
vertices. Moreover,
\[
 b(G_{r,t},H_{r,t})\ge L_{r,t}
   =(r-1)t+\frac{r(r+1)}2-1,
\]
and therefore
\[
 \liminf_{t\to\infty}\frac{b(G_{r,t},H_{r,t})}{n_{r,t}}\ge1-\frac1r.
\]
\end{theorem}
\begin{proof}
The unique permutation matching the two intact size vectors is the odd transposition
$\tau=(1\;2)$, so \cref{lem:blowup} gives $G_{r,t}\not\cong H_{r,t}$.

For $j=2$, delete from $P_2$ in $G_{r,t}$ and from $P_1$ in $H_{r,t}$; the two resulting vectors agree
and yield $t+2$ common cards. For $j\in\{3,\ldots,r\}$, delete from $P_j$ on both sides. Positions
$j-1$ and $j$ in the first card then have equal size, so $\rho_j=(j-1\;j)$ is harmless. The product
$\tau\rho_j$ is even and, by \cref{lem:blowup}, gives an isomorphism of the cards. Thus
\[
 b(G_{r,t},H_{r,t})\ge\sum_{j=2}^r(t+j)
 =(r-1)t+\frac{r(r+1)}2-1.
\]
The port classes contain $rt+r(r+1)/2$ vertices and the selector contributes $C_r$, proving the order.
Connectivity is immediate. Port degrees and selector degrees agree in both graphs; pair-vertex degrees
are $a_i+a_j+\lambda$, whose multiset is unchanged when the size vector is permuted. Hence the degree
sequences agree. Finally, with $r$ fixed,
\[
 \frac{L_{r,t}}{n_{r,t}}
 =\frac{(r-1)t+O_r(1)}{rt+O_r(1)}\longrightarrow1-\frac1r.
\]
\end{proof}

\begin{corollary}\label{cor:limsup}
For every $\varepsilon>0$, infinitely many orders $n$ admit connected nonisomorphic graphs with the same
degree sequence and more than $(1-\varepsilon)n$ common cards. Equivalently,
\[
 \limsup_{n\to\infty}\frac{B(n)}n=1.
\]
\end{corollary}
\begin{proof}
Choose an even $r$ with $1/r<\varepsilon$ and let $t\to\infty$ in \cref{thm:main}. The opposite
inequality $B(n)/n\le1$ is immediate.
\end{proof}

\section{Computational verification}\label{sec:experiment}
The proofs above give explicit lower bounds without computation. As a separate check, we generated both
complete vertex decks for several compact selector implementations and canonically labeled every
uncolored card with nauty's \texttt{labelg -S} \citep{MP14}. Deck intersections were computed as
multiset intersections of canonical labels. An independent pure-Python generator followed by the
standalone nauty executable reproduced all decks.

In \cref{tab:checks}, $Q(n)=2\lfloor(n-1)/3\rfloor$ is the overlap achieved by the BBF construction and
claimed as the eventual maximum by the Strong Reconstruction Conjecture. The ``excess'' column is the
verified overlap minus $Q(n)$, so every positive entry is a finite violation of the BBF formula.

\begin{table}[h]
\centering
\caption{Complete-deck verification with nauty. Each exact overlap equals the combinatorial lower bound
for that implementation; no auxiliary-vertex deletion contributes an additional common card.}
\label{tab:checks}
\begin{tabular}{@{}rrrrrr@{}}
\toprule
ports & size parameter & order & verified overlap & BBF benchmark & excess \\
$r$ & $t$ & $n$ & $b(G,H)$ & $Q(n)$ & $b-Q$ \\
\midrule
4 & 14  & 78  & 51  & 50  & 1  \\
5 & 10  & 81  & 54  & 52  & 2  \\
5 & 100 & 531 & 414 & 352 & 62 \\
6 & 27  & 234 & 155 & 154 & 1  \\
\bottomrule
\end{tabular}
\end{table}

The first row is the pair in \cref{thm:78}. The five-port checks use a compact odd-port selector with
the same deletion-cancellation mechanism. The six-port check uses the smaller alternative seed
$C_5\sqcup K_1$; it is not the six-port selector chosen in \cref{sec:general}. These computations verify
the stated finite instances, while the lower bounds in \cref{thm:78,thm:main} are proved directly.

\section{Comparison with the BBF benchmark}\label{sec:plot}
A comparison of raw card counts can hide the main effect, so \cref{fig:comparison} plots the guaranteed
fraction $L_{r,t}/n_{r,t}$ against the order on a logarithmic scale. For the selectors used in
\cref{thm:main}, eliminating $t$ gives
\[
 \frac{L_{4,t}}{n_{4,t}}=\frac34-\frac{7.5}{n},\qquad
 \frac{L_{6,t}}{n_{6,t}}=\frac56-\frac{60}{n},\qquad
 \frac{L_{10,t}}{n_{10,t}}=\frac9{10}-\frac{13644}{n},
\]
at their respective admissible orders. The BBF ratio $Q(n)/n$ fluctuates slightly because of the floor
function and tends to $2/3$. Larger $r$ gives a higher limiting fraction, but also a larger fixed
selector, which explains why the $r=10$ curve starts much later.

\begin{figure}[h]
\centering
\begin{tikzpicture}[font=\small]
\begin{scope}[x=1cm,y=15cm]
  \draw[->] (-.35,0) -- (7.35,0) node[right] {$n$ (log scale)};
  \draw[->] (-.35,0) -- (-.35,.325) node[above] {$L/n$};
  \foreach \x/\lab in {0/$10^2$,1/$10^3$,2/$10^4$,3/$10^5$,4/$10^6$,5/$10^7$,6/$10^8$,7/$10^9$}
    \draw[gray!25] (\x,0)--(\x,.31) (\x,.004)--(\x,-.004) node[below=3pt] {\lab};
  \foreach \y/\lab in {0/.60,.05/.65,.10/.70,.15/.75,.20/.80,.25/.85,.30/.90}
    \draw[gray!25] (-.35,\y)--(7.15,\y) (-.31,\y)--(-.39,\y) node[left=3pt] {\lab};

  \draw[blue,thick] plot coordinates {
(-0.3010,0.040000) (-0.2596,0.054545) (-0.2218,0.033333) (-0.1871,0.046154)
(-0.1549,0.057143) (-0.1249,0.040000) (-0.1079,0.041026) (-0.0969,0.050000)
(-0.0706,0.058824) (-0.0458,0.044444) (-0.0223,0.052632) (0.0000,0.060000)
(0.0086,0.047059) (0.0414,0.054545) (0.0792,0.050000) (0.1303,0.051852)
(0.1761,0.053333) (0.2430,0.062857) (0.3010,0.060000) (0.3054,0.063366)
(0.3979,0.064000) (0.4771,0.060000) (0.6021,0.065000) (0.6990,0.064000)
(0.7007,0.065339) (0.8751,0.064000) (1.0000,0.066000) (1.0009,0.064671)
(1.3010,0.066000) (1.6990,0.066400) (2.0000,0.066600) (3.0000,0.066660)
(4.0000,0.066666) (5.0000,0.066667) (6.0000,0.066667) (7.0000,0.066667)};
  \draw[orange!90!black,very thick] plot coordinates {
(-0.1079,0.053846) (0.0086,0.076471) (0.3054,0.112871) (0.7007,0.135060)
(1.0009,0.142515) (2.0001,0.149250) (3.0000,0.149925) (4.0000,0.149993)
(5.0000,0.149999) (6.0000,0.150000) (7.0000,0.150000)};
  \draw[green!55!black,very thick] plot coordinates {
(0.4116,0.000775) (0.4771,0.033333) (0.7782,0.133333) (1.0009,0.173453)
(2.0001,0.227335) (3.0000,0.232733) (4.0000,0.233273) (5.0000,0.233327)
(6.0000,0.233333) (7.0000,0.233333)};
  \draw[red!75!black,very thick] plot coordinates {
(2.6990,0.027120) (3.0000,0.163560) (3.3010,0.231780) (4.0000,0.286356)
(5.0000,0.298636) (6.0000,0.299864) (7.0000,0.299986)};
  \fill[orange!90!black] (-.1079,.053846) circle (.055cm);
  \draw[->] (.45,.105)--(-.07,.061);
  \node[anchor=west,fill=white,inner sep=1pt] at (.45,.108) {$51/78$};

  \draw[blue,thick] (4.25,.115)--(4.75,.115);
  \node[anchor=west] at (4.85,.115) {BBF $Q(n)/n$};
  \draw[orange!90!black,very thick] (4.25,.095)--(4.75,.095);
  \node[anchor=west] at (4.85,.095) {$r=4$};
  \draw[green!55!black,very thick] (4.25,.075)--(4.75,.075);
  \node[anchor=west] at (4.85,.075) {$r=6$};
  \draw[red!75!black,very thick] (4.25,.055)--(4.75,.055);
  \node[anchor=west] at (4.85,.055) {$r=10$};
\end{scope}
\end{tikzpicture}
\caption{Guaranteed common-card fraction for three fixed-$r$ families from \cref{thm:main}, compared
with the BBF benchmark. Points on a fixed-$r$ curve occur only at its admissible integer orders; line
segments are visual guides. The BBF curve uses the exact floor-function value $Q(n)/n$.}
\label{fig:comparison}
\end{figure}
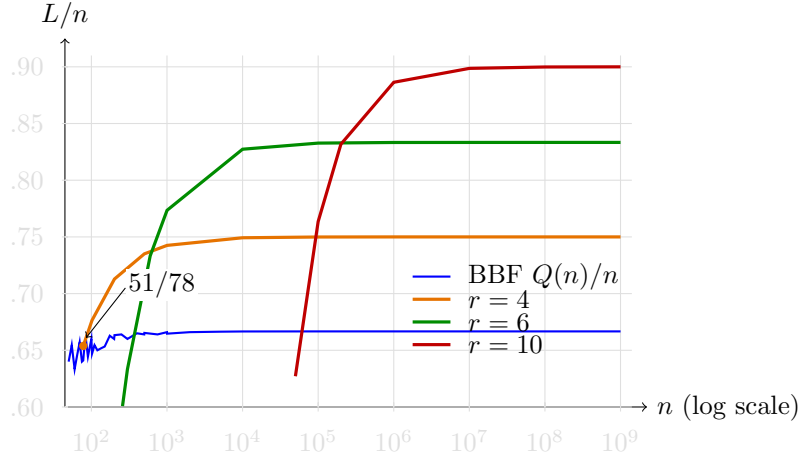

\section{Related work}\label{sec:related}
The classical Reconstruction Conjecture and its methods are surveyed by \citet{BH77}; reconstruction
numbers and the ally/adversary viewpoint are surveyed by \citet{AFLM10}, following work including
\citet{Myrvold88}. The closest direct predecessor is \citet{BBF10}, who introduced the unrestricted
benchmark $Q(n)$ and also constructed high-overlap pairs with equal degree sequences. For restricted
pairs, \citet{BBFM11} proved that a connected and a disconnected $n$-vertex graph have at most
$\lfloor n/2\rfloor+1$ common cards. Automorphism-based common-card arguments also appear in the
supercard framework of \citet{BF18super}.

Incomplete-deck results recover particular invariants rather than the entire graph. In particular,
\citet{BF18size} proved that, for $n\ge29$, the number of edges is reconstructible from any deck missing
two cards. \citet{GGS21} strengthened size reconstructibility for large graphs to a growing number of
missing cards. These results are compatible with our examples, whose two roots already have the same
entire degree sequence. Recent restricted-class work includes \citet{Zwaneveld23} on recognizing trees,
forests, girth, and bipartiteness from incomplete decks. None of these invariant-recovery or
restricted-class bounds gives an unrestricted upper bound below the overlap constructed here.

\section{Conclusion}\label{sec:conclusion}
The $78$-vertex pair gives a small explicit violation of the BBF benchmark, while the fixed-$r$ families
show that the attainable overlap fraction is arbitrarily close to $1$. This work does \emph{not} prove or
disprove the ordinary Reconstruction Conjecture: every constructed pair still has cards outside the
explicitly matched subdeck. The construction may instead help identify new graph families that share
increasingly many common cards and clarify which information in a nearly complete deck can still
distinguish nonisomorphic graphs.

\end{document}